\documentclass[12pt]{article}
\usepackage[affil-sl]{authblk}
\usepackage[T1]{fontenc}
\usepackage[latin2]{inputenc}
\usepackage{amsmath,amssymb,amsmath,amsthm}
\newtheorem{lemma}{Lemma}[section]

\newtheorem{definition}[lemma]{Definition}

\newtheorem{remark}[lemma]{Remark}

\usepackage[dvipsnames]{color}

\def\Z{\mathbb Z}

\def\R{\mathbb R}
\def\C{\mathbb C}

\def\id{\mathrm{id}} 
\def\be{\begin{equation}}
\def\ee{\end{equation}}
\title{Twisted reality condition for spectral triple \\ on two points}
\author{Ludwik D\k{a}browski${}^1{}^*$}
\author{Andrzej Sitarz${}^{2,3}$\thanks{L.D\k{a}browski and A.Sitarz acknowledge 
support of the NCN grant 2012/06/M/ST1/00169}}
\affil{${}^1$ SISSA (Scuola Internazionale Superiore di Studi Avanzati), Via Bonomea 265, 34136 Trieste, Italy, \\
\it E-mail: {dabrow@sissa.it} 
}
\affil{${}^2$ Institute of Physics, Jagiellonian University,\\
prof. Stanisława {\L}ojasiewicza 11, 30-348 Krak\'ow, Poland,}
\affil{${}^3$ Institute of Mathematics of the Polish Academy of Sciences, \\ 
\'Sniadeckich 8, Warszawa, 00-950 Poland. \\
\it E-mail: {andrzej.sitarz@uj.edu.pl}  
}

\begin{document}
\maketitle

\begin{abstract}Lowest dimensional spectral triples with twisted reality condition over the function 
algebra on two points are discussed.  The gauge perturbations (fluctuations), 
chiral gauge perturbations, conformal rescalings, and permutation of the two points are 
presented.
\end{abstract}

\hbox{2010 Mathematics Subject Classification: {58B34, 58B32, 46L87}}

\section{Introduction}

Real spectral triples should be understood as a noncommutative generalisation of spin manifolds, as
proposed by Connes in \cite{Connes:1995} and demonstrated later in the reconstruction theorem
\cite{Connes:2013}. In a recent paper \cite{BCDS16}, motivated by some examples of conformally rescaled noncommutative Dirac operators as well as examples arising from the quantum deformations we have proposed a new
definition of reality condition, which includes a twist by an automorphism of the algebra\footnote{
remaining in the framework of {\em bona fide} spectral triples in contrast to twisted spectral triples 
of the interesting recent paper \cite{LaMa}}.

In this note we study the applications and consequences of the proposed definition for the finite
spectral triple over the algebra of functions on two points, which is often taken as the simplest
toy model for an almost noncommutative geometry and used to show the principles which are
behind the noncommutative description of the Standard Model as arising from a discrete noncommutative
geometry.
 
We will discuss systematically the lowest dimensional representations and the space of all Dirac operators 
with twisted and untwisted reality condition satisfied, due to the gauge perturbations (fluctuations), 
chiral gauge perturbations, conformal rescalings, and permutation of the two points. The formula for the distance between the two points will be presented.
 
\section{Twisted reality}

Let us recall the notion of real spectral triples with the twisted first order condition.
\begin{definition}\label{def.reality}
Let $\nu\in Aut(H)$ be a selfadjoint invertible operator on  a Hilbert space $H$.
Let $A$ be a complex $*$-algebra of operators on $H$,  which is left invariant by the 
automorphism $\nu$, i.e. $\nu^{-1} a \nu\in A$ for any $a\in A$. 
We say that the spectral triple $(A,H,D)$ admits 
a {\em $\nu$-twisted real structure} if there exists an anti-linear map $J: H\to H$ such 
that $J^*J=\id$, $J^2 = \epsilon\, \id$, and, for all $a,b\in A$,
\begin{equation}\label{o0c}
 [a, JbJ^{-1}] =0, 
\end{equation}
\begin{equation}\label{to1c}
 [D,a] J {\nu}^{-2}b\,{\nu}^{2}J^{-1} = J b\,J^{-1}[D,a], 
\end{equation}
\begin{equation}\label{tc}
DJ\nu = \epsilon' \nu JD, 
\end{equation}
\begin{equation}\label{treg}
\nu J\nu = J, 
\end{equation}
where $\epsilon,\epsilon' \in \{+1,-1\}$. 

If $(A,H,D)$ admits a grading operator 
$\gamma: H\to H$, $\gamma^*=\gamma$, $\gamma^2 = \id$, $[\gamma,a] =0$, for all $a\in A$, $\gamma D = -D\gamma$, 
and $\nu^2 \gamma = \gamma\nu^2 $,
then the twisted real structure $J$ is also required to satisfy
\begin{equation}\label{gc}
\gamma J = \epsilon''J\gamma,
\end{equation}
where $\epsilon''$ is another sign.
\end{definition}

This is a particular instance of Definition 2.1 in \cite{BCDS16}. 
The signs $\epsilon, \epsilon',\epsilon''$ determine 
the $KO$-dimension modulo 8 in the usual way 
\cite{Connes:1995} and the operator $J$ is antiunitary.
In such a case we shall say that a spectral triple 
admits a {\em $\nu$-twisted real structure},
or simply that it is a {\em $\nu$-twisted real spectral triple}. 

Often \eqref{o0c} is called the {\em order-zero condition}, \eqref{to1c} is called the {\em twisted order-one condition},  while  we shall refer to \eqref{tc} as to the {\em twisted $\epsilon'$-condition}, and to \eqref{treg} as to the {\em twisted regularity} 
(with the adjective twisted omitted if $\nu = \hbox{id}$).

Let $\Omega^1_D$ be a bimodule of one-forms:
$$ \Omega^1_D :=\left\{ \sum_i \pi(a_i)[D,\pi(b_i)], \,\,a_i,b_i\in A \right\}, $$
where the sum is finite. 
By a fluctuated Dirac operator $D_\alpha$ we mean
$$ 
D_\alpha:= D + \alpha + \epsilon' \nu J \alpha J^{-1} \nu,
$$
with the requirement that $\alpha $
is selfadjoint. 
We shall often use the shorthand notation $\alpha '=  \nu J \alpha J^{-1} \nu$.
As shown in \cite{BCDS16}, $(A,H,D_\alpha)$ with (the same) $J$ is also a $\nu$-twisted real spectral triple. 
If $(A,H,D)$ is even with grading $\gamma$,
then $(A,H,D_\alpha)$ is even with (the same) 
grading $\gamma$.
Furthermore 
$$\Omega^1_{D_\alpha}=\Omega^1_D $$
and the twisted fluctuations with composition  form a semigroup. 
They correspond to "gauge transformations" in physics.

Note that it is also possible to modify the Dirac operator by a chiral gauge perturbation of the form:
\be\label{chirg}
D_A^\gamma = D + \gamma A + \epsilon' J \gamma AJ^{-1}, \ee
where $A\in \Omega^1_D$ is antihermitian. 

Another type of transformations are conformal rescalings by a positive element taken originally \cite{CoTr} from the algebra $A$, or \cite{DaSi15}
from $J AJ^{-1}$.
Here we take $k \in A$ to be not only positive but also  
invertible and such that $k^{-1}$ is also bounded. Denote $k_J := J k J^{-1}$.
As shown in \cite{BCDS16}, given a real spectral triple
$(A,H,D,J)$, for
$$D_{k_J} =  k_J D k_J, \qquad \nu(h) = (k^{-1} k_J) \, (h),$$
the datum $(A, H, J, D_{k_J}, \nu)$ is a $\nu$-twisted real spectral triple.
If furthermore $(A,H,D,J)$ is even with grading $\gamma$,
then $(A,H,D_{k_J}, J, \nu)$ is even with (the same) grading $\gamma$,
and has the same KO-dimension  as $(A, H, D, J)$. 
Furthermore, if  $(A,H,D, J, \nu)$ is a $\nu$-twisted real spectral triple which satisfies the 
twisted order-one condition, then, for all $k$ as above
such that $\bar{\nu}(k k_J) = k k_J$,  $(A, H, D_{k_J}, J, \mu)$ is a $\mu$-twisted real spectral triple, 
where 
$$
D_{k_J} =  k_J D k_J, \qquad \mu(h)  =  k_J \nu k^{-1}\, (h).
$$
The grading $\gamma$, if exists, is again unchanged.

We also note that 
$$\Omega^1_{D_{k_J}}=\Omega^1_D ,$$
and conformal twist by $k$ of gauge fluctuated by $A$  (untwisted) real spectral triple is 
a gauge fluctuation by $B=k^{-1}Ak^{-1}$
 of conformally twisted by $k$ spectral triple. In particular, with $\nu$ as in \eqref{dist}, we have
for the corresponding Dirac operators  
$$
D_{k_J}+A + \epsilon' \nu JAJ^{-1} \nu =  (D+ B + \epsilon' JBJ^{-1})_{k_J},
$$
provided that $A = kBk$. \\

Before we proceed with the simplest examples of finite spectral triples let us recall and important notion
of irreducibility of a (real) spectral triple. 
\begin{definition}
A real spectral triple over an algebra $A$ is irreducible if the representation of algebra 
generated by $\gamma, a, [D,b]$, $a,b \in  A$ is not reducible.
\end{definition}
Note that this definition obviously implies irreducibility of all the data (a weaker condition).
In particular, it is easy to find examples of classical spectral triples over manifolds, 
which are irreducible with $J$ included but are reducible in the sense of the definition above.
In this case $H$ would correspond to "charge $n$ plus $-n$" fields while in our case to "charge zero" fields.  

In this paper we shall study the simplest (nontrivial) possible example, which is the algebra of functions over two points. 
We shall work with its faithful representation on a Hilbert space, or what is the same with its isomorphic copy of operators.
We will be interested whether it admits an irreducible real spectral triple and discuss the space of all Dirac operators 
with twisted and untwisted reality condition satisfied. 
 
\section{A spectral triple over two points}

The $*$-algebra $A_2$ of complex valued functions over two points is 
isomorphic to $\C^2=\{(c_+,c_-)\}$ with $*$ acting as the complex conjugation.
We will use the vector space basis given by the projection $e=e^*=e^2=(1,0)$ and $1-e$,
where $1$ is the unity (identity) element. 
An arbitrary element of the algebra can be written as
$$ a = c_+ \, e + c_- \, (1-e).$$

Before we start with the construction of spectral triples let us recall the following fact:

\begin{lemma}
The only nondegenerate first order differential calculus over $A_2$ is 
the universal differential calculus, with:
\begin{equation}
 d \left( c_1 \, e + c_2 \, (1-e) \right) = c_1 \, de - c_2 de, 
\end{equation}
with the bimodule of one-forms generated by $de$. 
\begin{equation}
e \, de = de \, (1-e). \label{ede}
\end{equation}
The involution on the algebra extends to the one forms: $de^* = -de$.
\end{lemma}

As a corollary, for any given $D$, we have
\begin{equation}\label{eDe}
e[D,a]=[D,a](1-e), \quad \forall  a\in A.
\end{equation}

Furthermore a general one-form $A\in \Omega^1_D$ can be parametrized 
by two complex numbers $\phi_1, \phi_2$ as follows 
$$ A = \left(\phi_1 e + \phi_2 (1 - e)\right) [D,e] 
$$
The selfadjoint $A=A^*$ correspond to $\phi_2 = -\phi_1^*$.

Note also that for $a = c_+ e + c_- (1-e)$ we have:
\be\label{[D,a]}
[D, a] = (c_+-c_-)[D, e]. 
\ee
Recalling the expression for the distance between the two points in terms of the spectral data:
\be\label{dist}
d_D:=sup_{a\in A_2}\{|c_+-c_-|\,|\, \|[D,a]\|\leq 1 \}
\ee
from \eqref{[D,a]} we thus obtain that 
\be\label{dist1}
d_D= \frac{1}{\|[D,e]\|}.
\ee

We observe that the algebra $A_2$ has only one nontrivial automorphism corresponding 
to the permutation of the two points, $\nu(e) = 1-e$, 
so that $\nu^2 = \id$. Hence, the twisted order one condition \eqref{to1c} is identical to the 
usual order one condition. The same holds for the commutation relation of $\nu^2$ with $\gamma$.
 However the  $\epsilon'$-condition \eqref{tc} and the twisted regularity \eqref{treg} are different 
in the untwisted and  twisted case, which influences the  possible form of the other data.

\subsection{The minimal representation ($\C^2$)}

The lowest possible dimension of a faithful representation of the algebra $A_2$ is $2$. 
The canonical action (faithful representation) of $A_2$ on $\C^2$ is 
through diagonal matrices in $M_2(\C)$. Therefore $A_2$ is necessarily its own 
commutant and although there are several inequivalent possibilities for the real 
structure $J$, for all of them we necessarily have that $J a J^*$ is an element of the 
algebra itself. Therefore, the order one condition (\ref{to1c}) in its twisted or untwisted
version would contradict the bimodule relation between algebra elements and 
one-forms (\ref{eDe}). 

Hence, we reach the following conclusion:

\begin{lemma}
There are no irreducible real spectral triples (twisted or untwisted) with the
Hilbert space $H=\C^2$ over the algebra $A_2=\C^2$,
that give a non-zero differential calculus.
\end{lemma}

Of course, if we discard the real structure altogether we still have an irreducible
spectral triple.

\section{The minimal even spectral triples ($\C^3$)}

The first possibility to have a nontrivial real (even) spectral triple with a faithful representation of $A_2$, is 
on $\C^3$ with the  $\Z_2$-grading $\gamma$,  which we take in the diagonal form 
$$
\gamma = \left( \begin{array}{ccc} 
1 & 0 & 0  \\ 0 & -1  & 0 
\\ 0 & 0 & -1  
\end{array}  \right).
$$
and the representation of $A_2$:
$$ 
A_2 \ni a = \left( \begin{array}{ccc} c_+ & 0 & 0 \\   0 & c_+ & 0 
\\ 0 & 0 & c_- \end{array}  \right).
$$
The Dirac operator is a priori an arbitrary selfadjoint matrix in $M_3(\C)$, however, using
its anti-commutation relation with $\gamma$ we are restricted to the following form:

$$ 
 D  = \left( \begin{array}{cccc} 
 0 & d_2 & d_1 \\
 d_2^* & 0 & 0  \\
 d_1^* & 0 & 0 
 \end{array} \right),
$$
where $d_1, d_2 \in\C$.

\subsection{The distance formula}
Before we pass to real structures and the restrictions on $D$ imposed by 
\eqref{to1c} and \eqref{tc} we compute the one-forms $A \in \Omega^1_D$ associated to $D$. 
For that it suffices to compute
\be
\label{[D,e]3}
 [D,e] = \left( \begin{array}{ccc} 
 0 & 0 & -d_1  \\
 0 & 0 & 0 \\
 d_1^* & 0 & 0 
  \end{array} \right),
\ee
where 
$$ e =  \left( \begin{array}{ccc} 
   1 & 0 &  0 \\ 
   0 & 1  & 0  
\\ 0 & 0 & 0  
 \end{array}  \right). $$
Then it follows that a general one-form $A$ can be parametrized 
by two complex numbers $\phi_1, \phi_2$ as follows 
$$ A = 
 \left( \begin{array}{ccc} 
 0 & 0 & -\phi_1 d_1  \\
 0 & 0 & 0  \\
 \phi_2 d_1^* & 0 & 0  
 \end{array} \right) . 
$$
The selfadjoint $A=A^*$ correspond to $\phi_2 = -\phi_1^*$.

Since the norm of $[D,a]$ is
$$ || [D,a] || = |c_+ - c_-|\, |d_1| ,$$
it follows then that the distance \eqref{dist} between the two points equals 
\be\label{dist3}
d_D= \frac{1}{\|[D,e]\|}= \frac{1}{|d_1|}.
\ee

\subsection{Real structure}

It is easy to see that $[[D,e],e]\neq 0$ unless $d_1=d_2=0$. Thus to have the first 
order condition satisfied, $J$ cannot commute with $e$ and  we cannot have 
$J A_2 J\subset A_2$. 
The only (up to a simple rescaling) candidate for such $J$ is,
$$ 
J = \left( \begin{array}{ccc} 
  1 & 0 & 0 
\\ 0 & 0  & 1 
\\ 0 & 1 & 0  
 \end{array}  \right) \circ *.
$$
It satisfies $J^2=\id$ and $J \gamma = \gamma J$.
We note that $Ja^*J$ becomes
$$ 
J a^* J = \left( 
\begin{array}{ccc} 
c_+ & 0 & 0  \\ 0 & c_-  & 0 
\\ 0 & 0 & c_+  
\end{array}  \right).
$$

Now, concerning the order one condition, it suffices  to be checked only for 
$a=b=e$.  Indeed a simple computation shows that it is satisfied:
$$ [[D, e], JeJ^{-1}] = 0. $$
independently of any other requirements.

However, to have a family o real (possibly twisted) spectral triples, we need to investigate
possible twists.

\subsection{A family of real, conformally twisted, spectral triples.}

Let us begin by finding a family of Dirac operators compatible with $J$. 
We have:
$$ 
 D J  = 
 \left( \begin{array}{ccc} 
 0 & d_1 & d_3  \\
 d_3^* & 0 & 0  \\
 d_1^* & 0 & 0 
 \end{array} \right) \circ *, \quad \quad
 J D  = 
 \left( \begin{array}{ccc} 
 0 & d_3^* & d_1^*  \\
 d_1 & 0 & 0   \\
 d_3 & 0 & 0  
 \end{array} \right)  \circ *, $$
 so, imposing the commutation relation $JD = \epsilon' DJ$ (the $\epsilon'$-condition), 
we obtain:
 $$ d_3 = \epsilon' d_1^*. $$
 
The self-adjoint one-forms
\be\label{A3} 
A = \left(\phi e  - \phi^* (1 - e)\right) [D,e] = 
 \left( \begin{array}{ccc} 
 0 & 0 & -\phi d_1  \\
 0 & 0 & 0  \\
- \phi^* d_1^* & 0 & 0  
 \end{array} \right) 
\ee
allow to fluctuate the Dirac operator by a real perturbation:
$$ A + \epsilon' JAJ^{-1} = 
- \left( \begin{array}{ccc} 
 0 & \epsilon' \phi^* d_1^*  & \phi d_1 \\
\epsilon' \phi d_1 & 0 & 0 \\
 \phi^* d_1^* & 0 & 0 
 \end{array} \right).$$
 
Hence starting from the Dirac operator $D$ with parameters $d_1$, 
by fluctuations (gauge perturbations) we obtain a family of gauge 
perturbed Dirac operators $D+A + \epsilon' JAJ^{-1}$ parametrized by 
$(1-\phi) d_1$. It follows then from \eqref{dist} and \eqref{[D,e]3} that 
$$
d_{D+A + \epsilon' JAJ^{-1}}= \frac{1}{|1-\phi|}\, d_D.
$$

Concerning the perturbation of the type \eqref{chirg}
we have the following observation
\begin{remark}
The family of chiral gauge perturbations
$D+ \gamma A + \epsilon' J \gamma A J^{-1}$
with $A=-A^*\in \Omega^1_D$, 
is equal to the family of usual gauge perturbations of $D$. $\diamond$
\end{remark}

Now we shall consider a family of conformally rescaled Dirac operators. 
First of all, observe that there is no difference whether we use the rescaling
by an element from the algebra or an element from the commutant. 

\begin{lemma}
Let $k = \zeta ( \rho e + (1-\rho) (1-e) )$ and $h = \xi ( (1-\rho) e + \rho (1-e))$, where
$\zeta, \xi >0$ and $0\leq \rho \leq 1$.
Let $k_J = JkJ^{-1}$.
Then the conformal rescalings
$$ D_{k_J} = k_J D k_J, \qquad D_h = h D h, $$
 of $D$ are identical provided that $\zeta^2 \rho = (\xi)^2 (1-\rho)$ and they 
correspond to the same twist:
\be\label{nu3}
\nu = \left( \begin{array}{cccc} 
1 &  0 &  0 \\
0 &  \frac{1-\rho}{\rho} & 0  \\
0 & 0 & \frac{\rho}{1-\rho} 
 \end{array}  \right).
\ee
\end{lemma}

\begin{proof}
The proof  is by explicit computation. First:
$$ 
k_J  = \zeta \left( \begin{array}{ccc} 
 \rho & 0 & 0 \\
0 & (1-\rho) & 0  \\
0 & 0 & \rho  
 \end{array} \right), 
 \qquad
h =  \xi \left( \begin{array}{ccc} 
(1-\rho) & 0 & 0 \\
0 & \rho & 0  \\
0 & 0 & (1-\rho)  
 \end{array} \right), 
 $$
 out of which it is easy to see that indeed:
 $$
 D_{k_J}  = \zeta^2 \left( \begin{array}{ccc} 
 0 & \rho(1-\rho) \epsilon' d_1^* & \rho^2 d_1 \\
 \rho(1-\rho)\epsilon' d_1 & 0 & 0 \\
  \rho^2 d_1^* & 0 & 0  
 \end{array} \right), $$
 $$
 D_h  = (\xi)^2 \left( \begin{array}{ccc} 
 0 & (1-\rho)^2 \epsilon' d_1^* & \rho (1-\rho) d_1 \\
 (1-\rho)^2 \epsilon' d_1 & 0 & 0 \\
  \rho (1-\rho) d_1^* & 0 & 0  
 \end{array} \right).
 $$
The formula for the twists follows directly.
\end{proof}

As a result we can consider just one type of twists and using a general selfadjoint one-form \eqref{A3}  perform 
now fluctuations (gauge perturbations) of $D_{k_J}$ 
by 
$$ A + \epsilon' \nu JAJ^{-1} \nu = 
- \zeta^2 \left( \begin{array}{ccc} 
 0 & \epsilon'  \rho(1-\rho) \phi^* d_1^*  & \rho^2 \phi d_1 \\
\epsilon'  \rho(1-\rho) \phi d_1 & 0 & 0  \\
\rho^2 \phi^* d_1^* & 0 & 0  
\end{array} \right)$$
which are again parametrized by  $\phi\in \C$.

Since the overall factor $\zeta$ can be incorporated into a redefined parameter $\phi$, by the 
composition of  fluctuations (gauge perturbations)  with conformal rescalings of (untwisted) real 
nondegenerate (that is $d_1 \not= 0$) 
spectral triple, we obtain generically a 3-dimensional (real) parameter family of 
Dirac operators and thus of (twisted) spectral triples. 

We remark that similarly to the untwisted case the chiral gauge perturbations are identical with the usual ones.

Note also that the distance formula is only rescaled by $\rho \zeta$:
\be
\label{dist3}
d_{D_{k_J}}= \frac{1}{\|[D_{k_J},e]\|}= \frac{1}{\rho^2 \zeta^2 |d_1|},
\ee
however, the place where the twist appears is in the respective powers of
the Dirac operator:

$$
 (D_{k_J})^2  = \zeta^4 d_1 d_1^* \left( \begin{array}{ccc} 
  \rho^2 ((1-\rho)^2 + \rho^2) & 0 & 0\\
  0 & \rho^2 (1-\rho)^2 & 0 \\
  0 & 0 & \rho^4   
 \end{array} \right).
 $$

\subsection{The permutation twist}

In Definition \ref{def.reality} we adopted the usual assumption about the twist $\nu$ that $\nu^{-1} \pi(a) \nu$ is $\pi(\bar{\nu}(a))$, 
where $\bar\nu$ is an automorphism of the algebra. In our the case with the chosen representation there are
no such nontrivial maps $\nu$, however we can depart slightly from this assumption if 
the automorphism $\nu$ satisfies $\nu^2=1$.
In that case the relation (\ref{to1c}) is automatically  satisfied. If we take
$$
\nu = \left( \begin{array}{ccc} 
1 &  0 &  0 \\
0 & 0 & 1  \\ 
0 & 1 & 0  
 \end{array}  \right),
$$
then also the twisted regularity condition 
$\nu J \nu = J$
is satisfied.\\

Furthermore, we have:
$$ 
 D J \nu  = 
 \left( \begin{array}{ccc} 
 0 & d_2 & d_1 \\
 d_2^* & 0 & 0  \\
 d_1^* & 0 & 0   
 \end{array} \right) 
 \circ *, \quad \quad
 \nu J D  = 
 \left( \begin{array}{ccc} 
 0 &   d_2^* &  d_1^*  \\
 d_2 & 0 & 0  \\
 d_1 & 0 & 0  
 \end{array} \right) 
  \circ *, $$
 so, imposing the twisted $\epsilon'$-condition, $DJ \nu = \epsilon'\nu  JD$, we obtain:
$$ d_2^* = \epsilon' d_2, \quad \quad d_1^* = \epsilon' d_1. $$
Hence the family of $\nu$-real Dirac operators is parametrized 
by $(d_1,d_2) \in \R$ (for $\epsilon'=1$) or $(d_1,d_2) \in i\R$ (for $\epsilon'=-1$):
$$
\left( \begin{array}{cccc} 
 0 &  d_2 & d_1  \\
 \epsilon\rq{} d_2 & 0 & 0  \\
 \epsilon\rq{} d_1 & 0 & 0  
 \end{array} \right).
$$

One can similarly as before compute the family of gauge fluctuated Dirac operators, 
which is parametrized by a complex number $\phi$ and amounts to the change: 
$$(d_1,d_2) \to ((1-\phi - \phi^*) d_1, d_2).$$ 

The chiral gauge fluctuations are again exactly the same, so we see that only one 
parameter of the Dirac operator is changed via all possible gauge transformations.

The formula for the distance for the fluctuated $D$ is, 
$$
d_{D+A + \epsilon' \nu JAJ^{-1} \nu}= 
\frac{1}{\max\{|1-\phi - \phi^*| |d_1|, |d_2| \}}.
$$

\section{The spectral triples on $\C^4$}

The next possibility of a low-dimensional spectral triple, which is irreducible is with
of $A_2$ on $\C^4$ and with the $\Z_2$-grading $\gamma$ such that $A_2$ has an
irreducible representation on each eigenspace of $\gamma$:
$$
\gamma = \left( \begin{array}{cccc} 
1 & 0 & 0 & 0 \\ 0 & -1 & 0 & 0 
\\ 0 & 0 & -1 & 0 \\ 0 & 0 & 0 & 1  
\end{array}  \right),
\qquad 
A_2 \ni a = \left( \begin{array}{cccc} c_+ & 0 & 0 & 0\\   0 & c_+ & 0 & 0 
\\ 0 & 0 & c_- & 0 \\ 0 & 0 & 0& c_- \end{array}  \right).
$$
The Dirac operator is a priori an arbitrary selfadjoint matrix in $M_4(\C)$, however, using
its anti-commutation relation with $\gamma$ we are restricted to the following form:

$$ 
 D  = \left( \begin{array}{cccc} 
 0 & d_3 & d_1 & 0 \\
 d_3^* & 0 & 0 & d_2 \\
 d_1^* & 0 & 0 & d_4 \\
 0 & d_2^* & d_4^* & 0 
 \end{array} \right),
$$
where $d_1, d_2, d_3, d_4 \in\C$.

\subsection{The distance formula}
Before we pass to real structures and the restrictions on $D$ imposed by \eqref{to1c} and \eqref{tc} we compute 
the one-forms $A\in \Omega^1_D$ associated to $D$. For that it suffices to compute
\be\label{[D,e]}
 [D,e] = \left( \begin{array}{cccc} 
 0 & 0 & -d_1 & 0 \\
 0 & 0 & 0 & -d_2 \\
 d_1^* & 0 & 0 & 0 \\
 0 & d_2^* & 0 & 0 
 \end{array} \right),
\ee
where 
$$ e =  \left( \begin{array}{cccc} 
1 & 0 & 0 & 0 \\ 0 & 1  & 0 & 0 
\\ 0 & 0 & 0 & 0 \\ 0 & 0 & 0 & 0  
 \end{array}  \right), $$
together with $1$, forms the basis of $A_2$.
Then it follows that a general one-form $A$ can be parametrized by 
two complex numbers $\phi_1, \phi_2$ as follows 
$$ A = \phi_1 e [D, e] + \phi_2 (1 - e) [D,e] = 
 \left( \begin{array}{cccc} 
 0 & 0 & -\phi_1 d_1 & 0 \\
 0 & 0 & 0 & - \phi_1 d_2 \\
 \phi_2 d_1^* & 0 & 0 & 0\\
 0 & \phi_2 d_2^* & 0 & 0 
 \end{array} \right) . 
$$
The selfadjoint $A=A^*$ correspond to $\phi_1 = -\phi_2^*=:\phi $.

We also note that for $a = c_+ e + c_- (1-e)$:
\be [D, a] = (c_+-c_-)[D, e]. \ee
The norm of $[D,a]$ is
$$ || [D,a] || = |c_+ - c_-|\, \max\{ |d_1|, |d_2| \}, $$
It follows then that the distance \eqref{dist} between the two points equals 
\be\label{dist4}
d_D= \frac{1}{\|[D,e]\|}= \frac{1}{\max\{ |d_1|, |d_2| \}}.
\ee

\subsection{The real structure on $\C^4$.}

Again, as in the case previous case we look for possible $J$, antilinear isometries
that map the algebra into the genuine commutant.
Then, up to a unitary transform such $J$ takes the form
$$ 
J = \left( \begin{array}{cccc} 
  1 & 0 & 0 & 0 
\\ 0 & 0  & 1&  0 
\\ 0 & 1 & 0 & 0 
\\ 0 & 0 & 0 & 1  
 \end{array}  \right) \circ *.
$$
It satisfies $J^2=\id$ and $J \gamma = \gamma J$.
We note that $Ja^*J$ becomes
$$ 
J a^* J = \left( \begin{array}{llll} 
c_+ & 0 & 0 & 0 \\ 0 & c_-  & 0 & 0 
\\ 0 & 0 & c_+ & 0 \\ 0 & 0 & 0 & c_-  
 \end{array}  \right).
$$

Now, concerning the order one condition, it suffices  to be checked only for $a=b=e$. 
Indeed a simple computation shows that it is satisfied:
$$ [[D, e], JeJ^{-1}] = 0. $$

Next we shall impose the (possibly twisted) commutation relations 
\eqref{tc} and  \eqref{treg} separately for the untwisted and twisted
cases.

\subsection{Real spectral triples}

We begin with the untwisted case, first 
$$ 
 D J  = 
 \left( \begin{array}{cccc} 
 0 & d_1 & d_3 & 0 \\
 d_3^* & 0 & 0 & d_2 \\
 d_1^* & 0 & 0 & d_4 \\
 0 & d_4^* & d_2^* & 0 
 \end{array} \right) \circ *, \quad \quad
 J D  = 
 \left( \begin{array}{cccc} 
 0 & d_3^* & d_1^* & 0 \\
 d_1 & 0 & 0 & d_4^* \\
 d_3 & 0 & 0 & d_2^* \\
 0 & d_2 & d_4 & 0 
 \end{array} \right)  \circ *, $$
 so, imposing the commutation relation $JD = \epsilon' DJ$ (the $\epsilon'$-condition), 
we obtain:
 $$ d_3 = \epsilon' d_1^*, \quad \quad d_4 = \epsilon' d_2^*. $$
 
Having now fixed $d_1$ and $d_2$ we can consider the space of all real perturbations
of the Dirac operator by one forms. The associated space of one forms $A$ can 
be parametrized by two complex numbers $\phi_1, \phi_2$, however a selfadjoint 
one-form has $\phi_1 = -\phi_1^*=:\phi $, and we have

$$ A = \left( \phi e - \phi^* (1 - e)\right) [D,e] = 
- \left( \begin{array}{cccc} 
 0 & 0 & \phi d_1 & 0 \\
 0 & 0 & 0 &  \phi d_2 \\
 \phi^* d_1^* & 0 & 0 & 0\\
 0 & \phi^* d_2^* & 0 & 0 
 \end{array} \right) 
$$

and the real perturbation:
$$ A + \epsilon' JAJ^{-1} = 
- \left( \begin{array}{cccc} 
 0 & \epsilon' \phi^* d_1^*  & \phi d_1& 0 \\
\epsilon' \phi d_1 & 0 & 0 & \phi d_2 \\
 \phi^* d_1^* & 0 & 0 & \epsilon' \phi^* d_2^*\\
 0 & \phi^* d_2^* & \epsilon' \phi d_2& 0 
 \end{array} \right).$$
 
Hence starting from the Dirac operator $D$ with parameters $(d_1,d_2)$, 
by fluctuations (gauge perturbations) we obtain a family of gauge perturbed 
Dirac operators $D+A + \epsilon' JAJ^{-1}$ parametrized by 
$((1-\phi) d_1, (1-\phi) d_2)$.
It follows then from \eqref{[D,e]} and \eqref{dist4} that 
$$
d_{D+A + \epsilon' JAJ^{-1}}= \frac{1}{|1-\phi|}\, d_D.
$$
 
The chiral gauge perturbations \eqref{chirg} with $A=-A^*\in \Omega^1_D$, become
$$
\gamma A + \epsilon' J \gamma A J^{-1}=
- \left( \begin{array}{cccc} 
 0 & \epsilon' \phi^* d_1^*  & \phi d_1& 0 \\
\epsilon' \phi d_1 & 0 & 0 & -\phi d_2 \\
 \phi^* d_1^* & 0 & 0 & - \epsilon' \phi^* d_2^*\\
 0 & - \phi^* d_2^* & - \epsilon' \phi d_2& 0 
 \end{array} \right),
$$
and we see that starting from the Dirac operator $D$ with parameters 
$(d_1,d_2)$, the chiral fluctuations (gauge perturbations) lead to a family 
parametrized by $((1-\phi) d_1, -(1-\phi) d_2)$.

We can sum up the results of this subsection as
\begin{lemma}
In the real spectral triple over two points represented on $\C^4$
the entire family of Dirac operators ($D \neq 0$) is obtained through 
gauge fluctuations and chiral gauge fluctuations 
starting from a nondegenerate spectral triple ($d_1 \not= 0$).
\end{lemma}

\subsection{The family of twisted real spectral triples.}

Since the reality operator is in fact very similar to the one in the $\C^3$ case,
we find a similar property concerning conformal scaling by an element from
the algebra and the one from the commutant. 

\begin{lemma}
Let $k = \zeta ( \rho e + (1-\rho) (1-e) )$ and $h = \xi ( (1-\rho) e + \rho (1-e))$,
where
$\zeta, \xi >0$ and $0\leq \rho \leq 1$.
Let $k_J = JkJ^{-1}$.
Then the conformal rescalings 
$$ D_{k_J} = k_J D k_J, \qquad D_h = h D h, $$
of $D$ are identical provided that $\zeta^2 \rho = (\xi)^2 (1-\rho)$ and they 
correspond to the same twist:
$$
\nu = \left( \begin{array}{cccc} 
1 &  0 & 0 & 0 \\\rho
0 &  \frac{1-\rho}{\rho} & 1 & 0  \\
0 & 0 & \frac{\rho}{1-\rho} & 0 \\ 
0 & 0  & 0 & 1 
 \end{array}  \right).
$$
\end{lemma}

We skip the computational proof, and concentrate on the operator $D_{k_J}$, 
$$ 
 D_{k_J}  = \zeta^2 \left( \begin{array}{cccc} 
 0 & \rho(1-\rho) \epsilon' d_1^* & \rho^2 d_1 & 0 \\
 \rho(1-\rho)\epsilon' d_1 & 0 & 0 & (1-\rho)^2 d_2\\
  \rho^2 d_1^* & 0 & 0 &  \rho(1-\rho) d_2^*\epsilon' \\
 0 & (1-\rho)^2 d_2^* &  \rho(1-\rho) \epsilon' d_2 & 0 
 \end{array} \right)
 $$

We can perform now fluctuations (gauge perturbations) of $D_{k_J}$ by the real gauge 
fields one-forms, which are again parametrized by $\phi\in \C$:
$$ A + \epsilon' \nu JAJ^{-1} \nu = 
- \zeta^2 \left( \begin{array}{cccc} 
 0 & \epsilon'  \rho(1-\rho) \phi^* d_1^*  & \rho^2 \phi d_1& 0 \\
\epsilon'  \rho(1-\rho) \phi d_1 & 0 & 0 & (1-\rho)^2 \phi d_2 \\
\rho^2 \phi^* d_1^* & 0 & 0 & \epsilon'  \rho(1-\rho) \phi^* d_2^*\\
 0 & (1-\rho)^2  \phi^* d_2^* & \epsilon'  \rho(1-\rho) \phi d_2& 0 
 \end{array} \right).$$

We close this section by giving the norm of $[D_{k_J}+A + \epsilon' \nu JAJ^{-1} \nu,a]$
\be
|| [D_{k_J}+A + \epsilon' \nu JAJ^{-1} \nu,a] || = |c_+ - c_-|\, \zeta^2\, |1-\phi|\max\{\rho^2|d_1|,(1-\rho)^2|d_2| \} \}, 
\ee

It follows then that the distance between the two points is:
\be 
d_{D_{k_J}+A + \epsilon' \nu JAJ^{-1} \nu}=  \frac{1}{\|[D,e]\|}= \frac{1}{ \zeta^2 |1-\phi|\max\{\rho^2|d_1|,(1-\rho)^2|d_2| \}}.
\ee

\subsection{The permutation twist} 

We shall employ here the twist automorphism 
$$
\nu = \left( \begin{array}{cccc} 
0 &  0 & 0 & 1 \\
0 &  0 & 1 & 0  \\
0 & 1 & 0 & 0 \\ 
1 & 0  & 0 & 0 
 \end{array}  \right),
$$
which is involutive $\nu^2=\id$ and implements the automorphism that permutes (exchanges) $c_-$ and  $c_+$ in $a\in A$.

A simple calculation shows that the twisted regularity condition
$$
\nu J\nu = J
$$
is satisfied.

We remark that the twist automorphism 
$$
\nu' = \left( \begin{array}{cccc} 
0 &  0 & 1 & 0 \\
0 &  0 & 0 & 1  \\
1 & 0 & 0 & 0 \\ 
0 & 1  & 0 & 0 
 \end{array}  \right),
$$
that is also involutive $\nu^2=\id$ and implements the automorphism that permutes (exchanges) $c_-$ and  $c_+$ in $a\in A$, is not suitable as it does not satisfy the
twisted regularity condition. 

Furthermore, we have:
$$ 
 D J\nu  = 
 \left( \begin{array}{cccc} 
 0 & d_3 & d_1 & 0 \\
 d_2 & 0 & 0 & d_3^* \\
 d_4 & 0 & 0 & d_1^* \\
 0 & d_2^* & d_4^* & 0 
 \end{array} \right) \circ *, \quad \quad
 \nu J D  = 
 \left( \begin{array}{cccc} 
 0 &   d_2 &  d_4 & 0 \\
 d_3 & 0 & 0 & d_2^* \\
 d_1 & 0 & 0 & d_4^* \\
 0 & d_3^* & d_1^* & 0 
 \end{array} \right) 
  \circ *, $$
 so, imposing the twisted $\epsilon'$-condition,
  $DJ \nu = \epsilon'\nu  JD$, we obtain:
 $$ d_3 = \epsilon' d_2, \quad \quad d_4 = \epsilon' d_1. $$
Hence the family of $\nu$-real Dirac operators is: 
$$
\left( \begin{array}{cccc} 
 0 & \epsilon\rq{} d_2 & d_1 & 0 \\
 \epsilon\rq{} d_2^* & 0 & 0 & d_2 \\
 d_1^* & 0 & 0 & \epsilon\rq{} d_1 \\
 0 & d_2^* & \epsilon\rq{} d_1^* & 0 
 \end{array} \right).
$$

Similarly as in the untwisted case the family of gauge fluctuated
Dirac operators is parametrized by a complex number 
$\phi$ and amounts to the change: 
$$(d_1,d_2) \to ((1-\phi) d_1, (1-\phi) d_2).$$

The formula for the distance is 
$$
d_{D+A + \epsilon' \nu JAJ^{-1} \nu}= 
\frac{1}{|1-\phi|\max\{|d_1|,|d_2| \}}.
$$

\subsection{Composition of twists}

Finally, since in the case of the $A_2$ algebra we have two different types of twists, the conformal 
twists and the permutation twist (in the $\C^4$ representation case) there is a natural question,
whether these twists can be composed. However, explicit computations show:

\begin{lemma}
The composition of the conformal twists with the permutation twists does not satisfy
the twisted regularity condition (\ref{treg}).
\end{lemma}

\section*{Final remarks}

In this note a preparatory material has been presented for the study of more complicated finite 
dimensional examples, admitting non trivial twist automorphisms. Our interest is mainly in 
the finite spectral triple of the noncommutative version of the standard model, where 
the first order condition requires a particular attention. For that the product of (twisted) real spectral 
triples will need to be introduced and studied. It will be also interesting to study if the twist can 
be helpful with the reality condition satisfied only "up to infinitesimals" for some spectral triples 
on quantum groups. 

In this note, we have worked out for the simplest nontrivial algebra and lowest dimensional 
representations and spectral triples of the two-point space, which often serves as a toy model
for the finite part of Standard Model algebra (two-sheeted space). We have discussed the issue 
of the (usual and twisted) reality, conditions related to the conformal rescaling and to the permutation 
automorphism. The gauge perturbations (or "fluctuations") both usual and chiral have 
been explicitly computed and the dependence of the families of possible (twisted) Dirac
operators have been recast in terms of the fluctuation parameters.

We have demonstrated that even in the simplest possible example of a spectral triple, the twists
and the twisted reality conditions does indeed appear and need to be taken into account. These
preliminary results aim to demonstrated the richness of the the theory as well as serve as the basis
for future studies motivated by the noncommutative description of the elementary particle models.

%\section*{Acknowledgements}
%L. D{a}browski and A. Sitarz acknowledge support of the NCN grant 2012/06/M/ST1/00169. \\

\end{document}